\documentclass{amsart}
\usepackage{amssymb}
\usepackage{graphicx}
\usepackage{amsmath,amssymb}
\usepackage{amsfonts}
\usepackage{mathrsfs}
\newtheorem{theorem}{Theorem}[section]
\newtheorem{lemma}[theorem]{Lemma}
\newtheorem{proposition}[theorem]{Proposition}
\newtheorem{corollary}[theorem]{Corollary}

\theoremstyle{definition}
\newtheorem{definition}[theorem]{Definition}

\newcommand{\lang}{\begin{picture}(5,7)
\put(1.1,2.5){\rotatebox{45}{\line(1,0){6.0}}}
\put(1.1,2.5){\rotatebox{315}{\line(1,0){6.0}}}
\end{picture}\;}
\newcommand{\rang}{\;\begin{picture}(5,7)
\put(.1,2.5){\rotatebox{135}{\line(1,0){6.0}}}
\put(.1,2.5){\rotatebox{225}{\line(1,0){6.0}}}
\end{picture}}

\newcommand{\RR}{\mathbb{R}}

\newcommand{\NN}{\mathbb{N}}

\newcommand{\ZZ}{\mathbb{Z}}

\newcommand{\arr}{\longrightarrow}

\graphicspath{{immagini_articolo/}}

\usepackage[all,cmtip]{xy}
\usepackage{float}

\theoremstyle{remark}
\newtheorem{remark}[theorem]{Remark}

\numberwithin{equation}{section}

\author{Giovanni Alessandrini}
\address{Dipartimento di Matematica e Geoscienze, Universit\`a degli Studi di Trieste, Italy}

\email{alessang@units.it}
\thanks{The first author was supported  in part by Universit\`a degli Studi di Trieste, FRA 2016}

\author{Andrea Scapin}
\address{Dipartimento di Matematica e Geoscienze, Universit\`a degli Studi di Trieste, Italy}
\email{andrscapin@gmail.com}
\title[depth dependent resolution in eit]{depth dependent resolution in\\Electrical impedance tomography}
\begin{document}

\maketitle

\begin{abstract}
We consider the two-dimensional version of Calder\`on's problem. When the D-N map is assumed to be known up to an error level $\varepsilon_0$, we investigate how the resolution in the determination of the unknown conductivity deteriorates the farther one goes from the boundary. We provide explicit formulas for the resolution, which apply to conductivities which are perturbations, concentrated near an interior point $q$, of the homogeneous conductivity.
\end{abstract}

\section{Introduction}
We consider the well-known Calder\'{o}n's inverse boundary value problem, also known as Electrical Impedance Tomography. Given $K\ge 1$, and $\gamma \in L^{\infty}(\Omega)$, such that $K^{-1}\le \gamma\le K$, the so-called Dirichlet-to-Neumann map
\[ \Lambda_{\gamma}: H^{1/2}(\partial \Omega) \arr H^{-1/2}(\partial \Omega) \]
is the operator which associates to each $\varphi \in H^{1/2}(\partial \Omega)$ the conormal derivative $\gamma \partial_{\nu}u \in H^{-1/2}(\partial \Omega)$, where $u$ is the weak solution to the Dirichlet problem
\begin{equation}\label{basicDpb}
 \left\{
\begin{array}{llllll}
\text{div}( \gamma \nabla u) =0 \ ,&\text{in}& \Omega \ ,&\\
u = \varphi \ ,&\text{on} &\partial \Omega\ .&
\end{array}
\right.
\end{equation}
Calder\'{o}n's problem asks for the determination of $\gamma$, given $\Lambda_{\gamma}$ \cite{Ca}. We refer to Uhlmann \cite{U} for a thorough review on the progress and on the state of the art for this problem.

It is well-known that this problem is ill-posed \cite{A88, A07} and that, assuming a-priori regularity bounds of any order on $\gamma$, the best possible stability of $\gamma$ in terms of $\Lambda_{\gamma}$ is of logarithmic type, Mandache \cite{Man}. See also \cite{Fa} for the latest result of stability under minimal a-priori assumptions in the two--dimensional case, and for an updated reference list.

On the other hand, under minimal regularity assumptions, it is known that the boundary values of $\gamma$ depend in a Lipschitz fashion on $\Lambda_{\gamma}$,
\cite{Sy-U-2, A88, Bro}. It is then natural to ask how the determination  of the values of $\gamma$ does deteriorate the deeper we go inside the domain $\Omega$.

In this direction we mention the result of Nagayasu, Uhlmann and Wang \cite{NUW} who consider two--valued conductivities of the form
\[ \gamma = 1 + ( c -  1) \chi_{D} \ , \; c>0  \  ,\]
when $\Omega = B_R(0) \subset \mathbb R^2$ and the domain $D$ is a small perturbation of a disk $B_r(0)$, $0<r<R$. Examining the linearization $\mathrm{d}\Lambda$ of the corresponding Dirichlet-to-Neumann map, they show that the dependence of the  infinitesimal domain variation in terms of  $\mathrm{d}\Lambda$ deteriorates when $r\to 0$ at a logarithmic rate.    

Also in this note, we shall treat the two-dimensional setting,  but we shall consider more general perturbations of the homogeneous conductivity $\gamma_0\equiv  1$, and, rather than examining stability,   we shall discuss a more crude notion of \emph{resolution}. 

Let us briefly illustrate here our notion of resolution. Given an error level $\varepsilon_0>0$ on the Dirichlet-to-Neumann map, we shall say that two conductivities $\gamma_1, \gamma_2$ are \emph{indistinguishable} if $\|\Lambda_{\gamma_1}-\Lambda_{\gamma_2}\|_{*} \le \varepsilon_0$. Here $\|\cdot\|_*$ denotes the appropriate $H^{1/2}(\partial \Omega) \arr H^{-1/2}(\partial \Omega)$ norm. Next, fixing a disk $B_{\rho}(q)\subset \Omega$, we consider the class $\Gamma_\Omega (\rho, q)$ of conductivities which are perturbations of the reference homogeneous conductivity $\gamma_0\equiv  1$, and which may differ from $\gamma_0$ only inside $B_{\rho}(q)$. We shall call \emph{resolution limit} at level $\varepsilon_0$, for the point $q$, the largest $\rho>0$ such that all conductivities in $\Gamma_\Omega (\rho, q)$ are indistinguishable.

We recall that a related notion of \emph{distinguishability} has been already introduced by Isaacson and Cheney in \cite{Isa,IC}.

The main result of this note is the explicit calculation of such a resolution limit for all $q\in \Omega$ in two specific geometrical settings. Namely, when $\Omega$ is the unit disk $B_1(0)$ and when $\Omega$ is the half plane $\mathbb{H}^+$. Such explicit formulas illustrate that the resolution deteriorates as the distance from the boundary increases.

Our approach is based on few elementary facts.

(I) When $\Omega=B_1(0)$ the resolution limit for the center $q=0$ can be explicitly computed by separation of variables, \cite{A88}.

(II) The quadratic form $\lang \Lambda_{\gamma}\varphi, \varphi \rang = \int_{\Omega}\gamma |\nabla u|^2$
where $u$ and $\varphi$ are as in \eqref{basicDpb} is invariant under conformal mappings.

(III) The quadratic form $\lang \Lambda_{\gamma}\varphi, \varphi \rang$ above is monotone with respect to the conductivity $\gamma$. This is a well-known fact in the theory of EIT and has been used in many instances in the past \cite{A89,AR98,Ike,KSS, ARS00}.

(IV)  The explicit classical description in terms of M\"obius transformations of the automorphisms of the disk and of the conformal mappings of the half space onto the disk enable to reinterpret the formula for the resolution limit in each point in $B_1(0)$ or in $\mathbb{H}^+$.

In particular, we shall see that the case of the half plane is especially instructing, because in  this case the resolution limit depends linearly on the depth.

We wish to mention that, while this paper was in preparation, the authors became aware of the preprint by Garde and Knudsen \cite{GK} where similar considerations are developed. It may be noticed, however that the present approach has some differences.

i) In \cite{GK} only two-phase perturbations of the reference homogeneous conductivity are considered, whereas here we are able to treat  any variable perturbation.

ii) In \cite{GK}   the error on the data is evaluated with respect to  the   $L^{2}\arr L^{2}$  norm, instead of the  $H^{1/2}\arr H^{-1/2}$ norm, as we do here. This last choice, besides being physically motivated, has the fundamental advantage of being conformally invariant.   

iii) Here we examine the case of the half plane, which may be especially suggestive in connection to geophysical applications.

In the next section 2 we shall introduce the functional framework necessary for our analysis. The specific feature that we emphasize is that we are allowed to treat with equal simplicity bounded and unbounded simply connected domains in the plane. Next we show the basic conformal invariance of the functional spaces just introduced and of the Dirichlet-to-Neumann map. Finally we rigorously formulate the notions of indistinguishability and  of resolution limit.

In section 3 we compute the resolution limit. First we treat the case of the resolution limit at the center of the disk. Next we compute the resolution limit at an arbitrary point in the disk.
We examine the asymptotic behavior of the resolution limit with respect to the relevant parameters: depth, error level and ellipticity.
We conclude with the formulas for the  half plane.

\section{Preliminaries.}

We shall use the standard identification of $\mathbb{R}^2$ with  $\mathbb{C}$. Depending on the circumstances, points in the plane shall be represented by pairs $x = (x_1 , x_2)$ of real numbers or by a single complex number $z$.

Let $\Omega$ be a simply connected domain in $\mathbb{R}^2$, whose boundary is $C^{1, \alpha}$, $0 < \alpha < 1$.

\medskip

Let $K \ge 1$. Throughout the paper we shall consider conductivities $\gamma \in L^\infty (\Omega)$ which satisfy the following ellipticity condition:
\begin{equation} \label{eq:condizione-ellitticita} K^{-1} \le \gamma \le K .\end{equation}
In $H^1_{\text{loc}} (\Omega)$ we consider the equivalence relation: $u \sim v$ if and only if $u - v$ is constant. We define $H^1_{\Diamond}(\Omega)$ as the set of equivalence classes $[u]_{\sim}$ such that $u \in H^1_{\text{loc}} (\Omega)$ satisfying $\int_\Omega | \nabla u |^2 < \infty$. On $H^1_{\Diamond}(\Omega)$ we consider the norm given by
\[ \| [u]_\sim \|^2 = \int_\Omega | \nabla u |^2 .\]
From now on we shall simply write $u$ instead of $[u]_\sim \in H^1_\Diamond(\Omega)$. Let us remark that similar conventions have already been used, see for instance \cite[Section 16.1.2]{AIM}.

\medskip

\noindent
The corresponding trace space is defined as follows
\[ H^{1/2}_\Diamond (\partial \Omega) = H^1_\Diamond (\Omega)/ H_0^1(\Omega) . \]
On $H^{1/2}_{\Diamond}(\partial \Omega)$ we consider the norm given by
\[  \| \varphi \|_{H^{1/2}_{\Diamond} (\partial \Omega)} = \inf_{\substack{u \in H^1_\Diamond (\Omega) \\ u |_{\partial \Omega} = \varphi}} \| \nabla u \|_{L^2(\Omega)}.\]
Let us denote
\[ \| \varphi \|_{1/2} = \| \varphi \|_{H^{1/2}_{\Diamond} (\partial \Omega)}. \]
Let $\gamma \in L^\infty(\Omega)$ satisfying \eqref{eq:condizione-ellitticita}. Let $u \in H^1_\Diamond (\Omega)$ be the weak solution to
\begin{equation} \label{eq:problema-variazionale} \left \{ \begin{split} \text{div} (\gamma \nabla u) = 0 , & \:\text{ in } \Omega \\ u = \varphi , & \:\text{ on } \partial \Omega\end{split} \right .\end{equation}
where $\varphi \in H^{1/2}_\Diamond (\partial \Omega)$. 

\medskip

By the Riesz representation theorem it is clear that the solution to \eqref{eq:problema-variazionale} exists and it is unique.
\begin{definition} We denote by $H^{-1/2} (\partial \Omega)$ the dual space to $H^{1/2}_\Diamond (\partial \Omega)$ and we denote by  $\lang \cdot , \cdot \rang$ the $L^2(\partial \Omega)$-based duality between these spaces. Then we define the D-N map as follows
\[ \Lambda_\gamma : H^{1/2}_\Diamond (\partial \Omega) \arr H^{-1/2} (\partial \Omega) \]
for every $\varphi_1 , \varphi_2 \in  H^{1/2}_\Diamond (\partial \Omega) $
\begin{equation} \label{eq:triangolo} \lang \Lambda_\gamma \varphi_1 , \varphi_2 \rang  = \int_\Omega \gamma \nabla u_1 \cdot \nabla v_2 , \end{equation}
where $u_1$ is the solution to \eqref{eq:problema-variazionale} satisfying the boundary condition $\varphi = \varphi_1$ and $v_2$ is any function in $H^1_\Diamond (\Omega)$ satisfying $v_2 |_{\partial \Omega} = \varphi_2$.
\end{definition}

\begin{lemma}[Conformal invariance] \label{lem:lem_1} Let $\Omega, \Omega'$ be two simply connected domains whose boundaries are $C^{1,\alpha}$ and let  $\omega : \Omega' \arr \Omega$ be a conformal map between them. Let $\gamma$ satisfying \eqref{eq:condizione-ellitticita}. Then for all $\varphi_1 , \varphi_2 \in H^{1/2}_\Diamond (\partial \Omega)$  we have
\[ \lang \Lambda_{\gamma} \varphi_1 , \varphi_2 \rang \;=\; \lang \Lambda_{\gamma \circ \omega} \psi_1 , \psi_2 \rang , \]
where $\psi_i = \varphi_i \circ \omega$, $i = 1, 2$.
\end{lemma}

\begin{proof} Given $\varphi_1, \varphi_2 \in H^{1/2}_\Diamond (\partial \Omega)$ we consider $u_1 , u_2 \in H^{1}_\Diamond (\Omega)$ such that
\[ \left \{ \begin{split} \text{div} (\gamma \nabla u_i) = 0 ,& \:\text{ in } \Omega \\ u_i = \varphi_i , & \:\text{ on } \partial \Omega\end{split} \right . \]
Since $\partial \Omega , \partial \Omega'$ are $C^{1,\alpha}$, it is well-known that $\omega$ extends (with the same regularity) to  a diffeomorphism from $\overline{\Omega'}$ to $\overline{\Omega}$, $x = \omega(y)$. The Cauchy-Riemann equations can be written as follows
\[ \left ( \frac{\partial y}{\partial x} \right ) \left ( \frac{\partial y}{\partial x} \right )^T = \left | \text{det } \frac{\partial y}{\partial x} \right | I ,\] 
hence
\[ \begin{split} \lang \Lambda_{\gamma}  \varphi_1 , \varphi_2 \rang &= \int_{\Omega} \gamma(x) \: \nabla_x u_1  \cdot \nabla_x u_2 \text{ d} x \\ & = \int_{\Omega'} \gamma (\omega(y)) \: \frac{\left ( \frac{\partial y}{\partial x} \right )^T \nabla_y u_1 \left ( \frac{\partial y}{\partial x} \right )^T \nabla_y u_2}{\left |  \text{det } \frac{\partial y}{\partial x} \right |} \text{ d} y \\ & =  \int_{\Omega'} \gamma (\omega(y)) \: \nabla_y u_1 \cdot \nabla_y u_2 \text{ d} y \\ & = \lang \Lambda_{\gamma \circ \omega}  \psi_1 , \psi_2 \rang .
\end{split} \]
where $\psi_i = \varphi_i \circ \omega$, $i = 1,2$.
\end{proof}

\begin{corollary} \label{cor:corollary_1} Let $\varphi \in H^{1/2}_{\Diamond} (\partial \Omega)$. For any conformal map $\omega : \Omega' \arr \Omega$ we have
\[ \| \varphi \circ \omega \|_{H^{1/2}_{\Diamond} (\partial \Omega')} = \| \varphi \|_{H^{1/2}_{\Diamond} (\partial \Omega)}. \]
\end{corollary}

\begin{proof}We use Lemma \ref{lem:lem_1} with $\gamma \equiv 1$.\end{proof}

\begin{definition} Let us denote by $\| \cdot \|_*$ the $\mathscr{L} (H^{1/2}_{\Diamond} , H^{-1/2})$-norm, that is
\[ \| L \|_* = \sup_{\| \varphi \|_{1/2} = 1} \| L \varphi \|_{-1/2}. \]
\end{definition}
\begin{remark} We recall that if  $L : H^{1/2}_{\Diamond}(\partial \Omega) \arr H^{-1/2}(\partial \Omega) $ is selfadjoint then we also have
\[ \| L \|_* = \sup_{\| \varphi \|_{1/2} = 1} | \lang L \varphi , \varphi \rang | .\]
Hence this formula may be applied when $L = \Lambda$ is a D-N map and also when $L = \Lambda_1 - \Lambda_2$ is the difference of two D-N maps.
\end{remark}

\begin{corollary} \label{cor:cor_conformal-map-norm} Let $\gamma_{1} , \gamma_2$ be two conductivities in $\Omega$ and let $\omega : {\Omega'} \arr {\Omega}$ be a conformal map. Then
\[ \| \Lambda_{\gamma_1 \circ \omega} - \Lambda_{\gamma_2 \circ \omega} \|_* = \| \Lambda_{\gamma_1} - \Lambda_{\gamma_2}\|_* .\]
\end{corollary}
\begin{proof} Immediate consequence of Lemma \ref{lem:lem_1} and its Corollary \ref{cor:corollary_1}.\end{proof}

\begin{definition} We introduce the class
\[ \Gamma_\Omega (\rho, q) = \left \{ \gamma \in L^\infty (\Omega) \; : \; K^{-1} \le \gamma \le K , \; \gamma = 1 + \chi_{B_\rho(q)} ( \gamma - 1 ) \right \} \]
as the family of conductivities which are perturbations of the homogeneous conductivity $\gamma \equiv 1$, localized in $B_\rho (q)$. We shall call the point $q$ the \emph{center of the perturbation}.
\end{definition}

\begin{definition} Let $\varepsilon_0 > 0$ be the error level admitted on the known measurement of the map $\Lambda_\gamma$. We shall say that two conductivities $\gamma_1, \gamma_2$ are $\varepsilon_0$-\emph{indistinguishable} if \[ \| \Lambda_{\gamma_1} - \Lambda_{\gamma_2} \|_* \le \varepsilon_0 . \]
\end{definition}

\begin{definition} Given the disk $B_\rho (q)$, we denote two specific elements of $\Gamma_\Omega (\rho , q)$ as follows:
\[ \gamma_K = ( 1 + \chi_{B_\rho(q)} ( K - 1) ) , \]
\[ \gamma_{K^{-1}} = ( 1 + \chi_{B_\rho(q)} ( K^{-1} - 1) ). \]
\end{definition}
Note that for all $\gamma \in \Gamma_\Omega (\rho , q)$
\[ \gamma_K \le \gamma \le \gamma_{K^{-1}}. \]
For this reason it is sensible to call $\gamma_K , \gamma_{K^{-1}}$ the \emph{extreme conductivities} in $\Gamma_{\Omega} (\rho , q)$.

\begin{definition}  We define the \emph{resolution limit} (at level $\varepsilon_0$) relative to the center $q \in \Omega$ the number \[ \ell_q = \sup \left \{ \rho > 0  : \text{ for all } \gamma_1 , \gamma_2  \in \Gamma_{\Omega} ( \rho , q),\; \gamma_1 , \gamma_2 \text{ are indistinguishable} \right \}. \]
\end{definition}

For the sake of brevity, when $\rho, q$ are kept fixed, we denote by $\Lambda_i$ the map $\Lambda_{\gamma_i}$ and by  $\Lambda_K$, $\Lambda_{K^{-1}}$ the maps $\Lambda_{\gamma_{K}}, \Lambda_{\gamma_{K^{-1}}}$  rispectively.

\begin{lemma}  \label{lem:lem_monotonicity} Let $\gamma_1 , \gamma_2 \in \Gamma_\Omega (\rho,q)$. Then the following estimate holds:
\[ \| \Lambda_{1} - \Lambda_{2} \|_* \le  \| \Lambda_{K} - \Lambda_{K^{-1}} \|_* . \]
\end{lemma}
\begin{proof} Given $\gamma \in \Gamma_\Omega (\rho, q)$, by \eqref{eq:problema-variazionale}  for all $\varphi \in H^{1/2}_{\Diamond}( \partial \Omega)$ we have
\[ \lang \Lambda_{\gamma} \varphi , \varphi \rang = \inf_{\substack{u \in H^{1}_{\Diamond} (\Omega)\\u_{|\partial \Omega }= \varphi}} \int_\Omega \gamma | \nabla u |^2 . \]
Hence
\[ \lang \Lambda_{K^{-1}} \varphi , \varphi \rang \; \le \; \lang \Lambda_{i} \varphi , \varphi \rang \; \le \;\lang \Lambda_{K} \varphi , \varphi \rang \;,\quad i = 1,2 , \]
and consequently
\[ | \lang ( \Lambda_{1} - \Lambda_{2} ) \varphi , \varphi \rang | \le \; \lang (\Lambda_{K} - \Lambda_{K^{-1}}) \varphi , \varphi \rang .\]
\end{proof}

\begin{corollary} \label{cor:cor_1} 
\[ \ell_q = \sup \left \{ \rho > 0  : \gamma_K , \gamma_{K^{-1}}  \in \Gamma_{\Omega} ( \rho , q) \text{ are indistinguishable} \right \}. \]
\end{corollary} 
\begin{proof} Immediate consequence of Lemma \ref{lem:lem_monotonicity}. \end{proof}

\section{The resolution limit, formulas and asymptotics.}

\begin{lemma} \label{lemma:mappaK} For all $\varphi \in H^{1/2}(\partial B_1 (0))$, $\varphi(\theta) = \sum_{n \in \ZZ} \varphi_n e^{i n \theta}$, the following formulas hold:
\[ \begin{split}\Lambda_{K} \varphi &=  \sum_{n \in \ZZ} |n|  \dfrac{(K+1) + r^{2|n|} (K - 1)}{(K+1) - r^{2|n|} (K - 1)}  \varphi_n e^{i n \theta} ,\\
 \Lambda_{K^{-1}} \varphi &=  \sum_{n \in \ZZ} |n|  \dfrac{(K+1) - r^{2|n|} (K - 1)}{(K+1) + r^{2|n|} (K - 1)}  \varphi_n e^{i n \theta} . \end{split} \]
\end{lemma}
\begin{proof} By separation of variables in polar coordinates.\end{proof}

\begin{lemma} \label{lemma:fondamentale}
\begin{equation} \label{eq:eq_rappr} \| \Lambda_{K} - \Lambda_{K^{-1}} \|_* = \frac{4 (K^2 - 1) r^{2}}{(K+1)^2 - r^{4} (K-1)^2}  =  \frac{4 k r^{2}}{1 - k^2 r^{4}}, \end{equation}
where 
\begin{equation} \label{eq:k} k = \frac{K-1}{K+1} .  \end{equation}
\end{lemma}
\begin{proof} We have
\[ \begin{split} \| \Lambda_{K} - \Lambda_{K^{-1}} \|_* &= \sup_{\varphi \ne 0} \dfrac{ \lang (\Lambda_K - \Lambda_{K^{-1}} ) \varphi , \varphi \rang}{\| \varphi \|^2_{1/2}} \\ &=  \sup_{\varphi \ne 0}  \dfrac{\sum_{n \in \ZZ}  |n| \dfrac{4 (K^2 - 1) r^{2|n|}}{(K+1)^2 - r^{4|n|} (K-1)^2} |\varphi_n|^2}{\sum_{n \in \ZZ} |n| |\varphi_n|^2} . \end{split} \]
Since the expression \[ \dfrac{4 (K^2 - 1) r^{2|n|}}{(K+1)^2 - r^{4|n|} (K-1)^2} \] is decreasing with respect to $n \in \NN \setminus \{ 0 \}$, we obtain that
\[  \| \Lambda_{K} - \Lambda_{K^{-1}} \|_*  = \frac{4 (K^2 - 1) r^{2}}{(K+1)^2 - r^{4} (K-1)^2} .\]
\end{proof}

\begin{theorem}[The resolution at the center of a disk] \label{prop:prop_risol-bound} Let $\Omega = B_1(0)$. The resolution limit at the center of the disk $B_1(0)$ is
\begin{equation}\label{eq:eq_risol-bound} \ell_0 =   \sqrt{\frac{\sqrt{4 + \varepsilon_0^2} - 2}{\varepsilon_0  k}},\end{equation}
where $k$ is the constant introduced in \eqref{eq:k}.
\end{theorem}

\begin{proof} The extreme conductivities  $\gamma_K , \gamma_{K^{-1}} \in \Gamma_{B_1(0)} (r,0)$ are $\varepsilon_0$-indistinguishable if and only if 
\begin{equation} \label{eq:diseq1} \frac{4 (K^2 - 1) r^{2}}{(K+1)^2 - r^{4} (K-1)^2} \le \varepsilon_0 , \end{equation}
that is
\begin{equation} \label{eq:eq-bound} r \le \sqrt{ \frac{- 2 + \sqrt{4 + \varepsilon_0^2}}{\varepsilon_0  k}} . \end{equation}
Hence, by Corollary \ref{cor:cor_1}, the right-hand side in \eqref{eq:eq-bound} defines $\ell_0$.\end{proof}

\begin{remark} $\ell_0$ is meaningful only if $\ell_0 < 1$ and this corresponds to require \[ \varepsilon_0 < \varepsilon_{\max} = \frac{4 k}{1 - k^2}.\]
Evidently $\ell_0$ is an increasing function of $\varepsilon_0$ and Fig. \ref{fig:limdirisol0K3} examplifies its graph for a fixed value of $K$. 
\end{remark}

\begin{figure}[H]
\centering
\includegraphics[scale=0.40]{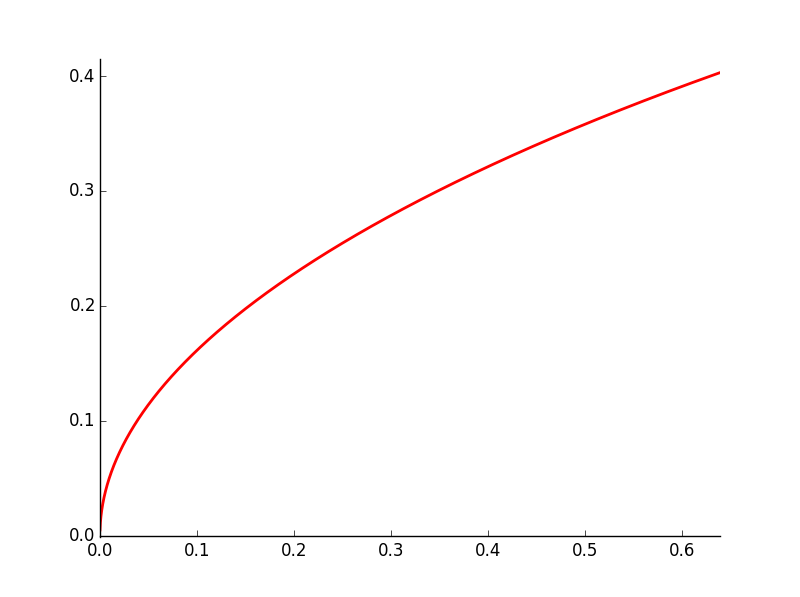} 
\caption{$\ell_0 = \ell_0 (\varepsilon_0 ,  K)$ con $K = 50$.\label{fig:limdirisol0K3}}
\end{figure}

Next we observe the following asymptotic behaviours  as function of $\varepsilon_0$ and $K$.

\begin{remark}
\begin{equation} \label{eq:comportamento-asintotico} \ell_0 (\varepsilon_0 , K) = \frac{1}{2 \sqrt{k}} \left ( \sqrt{\varepsilon_0} + O (\varepsilon_0^{5/2}) \right ) \quad \text{ as } \varepsilon_0 \to 0^+ \end{equation}
\end{remark}
Moreover, if we fix $\varepsilon_0 > 0$, we examine the behaviour with respect to $K$.
\[ \ell_0 (\varepsilon_0 , K) = C(\varepsilon_0) \sqrt{\frac{K+1}{K-1}}  ,\]
where 
\[ C(\varepsilon_0) = \sqrt{\frac{\sqrt{4 + \varepsilon_0^2} - 2}{\varepsilon_0}}.\]

\begin{remark} The function $\ell_0 = \ell_0 (\varepsilon_0 , K)$ has the following properties:
\begin{enumerate}
\item \( \displaystyle \lim_{K \to + \infty} \ell_0(\varepsilon_0 , K) = C(\varepsilon_0)  \) , \( \displaystyle \lim_{K \to 1^+} \ell_0(\varepsilon_0 , K) = + \infty \);
\item \( \ell_0 = \ell_0 (\varepsilon_0 , K) \) is strictly decreasing with respect to $K$;
\item  $\ell_0(\varepsilon_0 , K) < 1$ if $K > \dfrac{1 + C(\varepsilon_0) ^2}{1 - C(\varepsilon_0) ^2} = 2^{-1} \left ( \varepsilon_0 + \sqrt{ 4  + \varepsilon_0^2} \right )$.
\end{enumerate}
Note in particular that
\[ \inf_{K \ge 1} \ell_0 (\varepsilon_0 ,K) = C(\varepsilon_0) > 0.\]
Hence $C(\varepsilon_0)$ is a lower bound on the resolution limit which is independent of the ellipticity. See for example Fig. \ref{fig:lim_risol_0_Kinfty}  for $\varepsilon_0$ fixed at level $10^{-1}$.

\begin{figure}[H]
\centering
\includegraphics[scale=0.40]{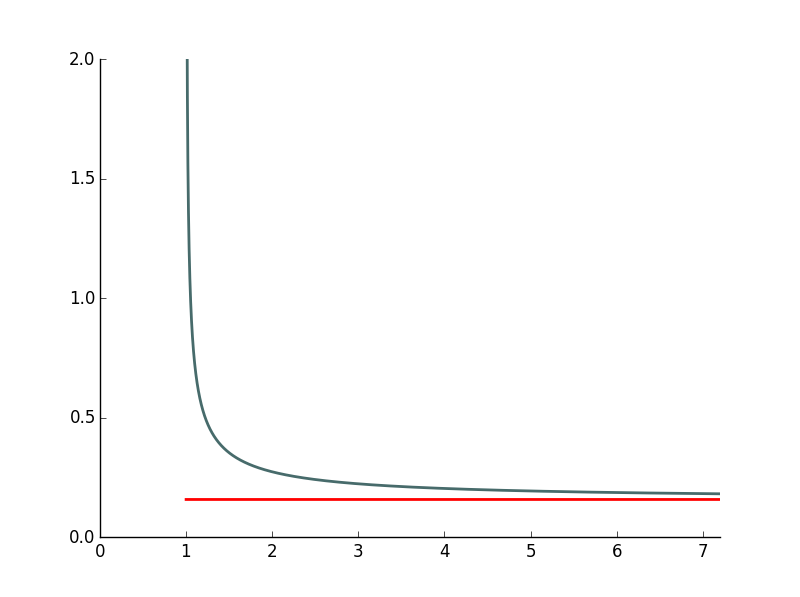} 
\caption{$\ell_0 = \ell_0 (\varepsilon_0 , K)$ with $\varepsilon_0 = 10^{-1}$ fixed.\label{fig:lim_risol_0_Kinfty}}
\end{figure}

Note also that if $K < 2^{-1} \left ( \varepsilon_0 + \sqrt{ 4  + \varepsilon_0^2} \right )$  then all conductivities are indistinguishable.
\end{remark}

\begin{proposition} \label{prop:disk_automorphism} Given $r \in (0,1)$ e $q \in [0,1)$, then there exists a (conformal) automorphism $f : \overline{B_1(0)} \longrightarrow \overline{B_1(0)}$ such that $f(B_\rho(q)) = B_r (0)$, where
\begin{equation} \label{eq:eq_relation_rho}  \rho = \frac{1 + r^2 - \sqrt{1 + (4 q^2 - 2) r^2 + r^4}}{2 r} . \end{equation}
\end{proposition}
\begin{proof} Up to rotations, the generic automorphism of $B_1(0)$ is given by
\[ f_p (z) = \frac{ z - p }{1 - p z} , \]
for any $p \in [0,1)$. We have
\[ | f_p(z) | = r  \quad \text{ if and only if } \quad \left | \frac{z - p}{1 - p z} \right |^2 =  r^2 .\]
That is: $f_p$ maps $B_r(0)$ onto $B_\rho (q)$ with $q, \rho$ given by
\[ \left \{ \begin{matrix} \displaystyle q = \frac{ p ( 1  -   r^2)}{1 - r^2 p^2 } , & \\ & \\ \displaystyle \rho = \frac{r ( 1 - p^2)}{1 - r^2 p^2}. \end{matrix} \right . \]
Viceversa, given $q$ and $r$, we can solve for $p$ and obtain
\[ \left \{ \begin{matrix} \displaystyle p = \sqrt{ \frac{1}{r^2} + \left ( \frac{1 - r^2}{2 r^2 q} \right )^2 } - \frac{1 - r^2}{2 r^2 q} ,& \\ & \\ \displaystyle \rho = \frac{1 + r^2 - \sqrt{1 + (4 q^2 - 2) r^2 + r^4}}{2 r} . \end{matrix} \right . \]
and \eqref{eq:eq_relation_rho} follows.
\end{proof}

\begin{theorem}[Depth dependent resolution in a disk]
Let $\Omega = B_1(0)$. The resolution limit at level $\varepsilon_0 > 0$, relative to the center $q$, is given by
\begin{equation} \label{eq:resolution_limit_disk_q} \ell_q =  \frac{1 + \ell_0^2 - \sqrt{ 1 + (4  q^2 - 2 ) \ell_0^2 + \ell_0^4}}{2 \ell_0}, \end{equation}
where $\ell_0$ is the number introduced in \eqref{eq:eq_risol-bound}.
\end{theorem}
\begin{proof} Straightforward consequence of Corollary \ref{cor:cor_conformal-map-norm}, Theorem \ref{prop:prop_risol-bound} and Proposition \ref{prop:disk_automorphism}.
\end{proof}

\begin{remark} We immediately see that
\[ \frac{\text{d}}{\text{d}q} \ell_q = - \frac{2 q \ell_0}{\sqrt{ 1 + (4  q^2 - 2 ) \ell_0^2 + \ell_0^4}} < 0 , \]
that is $\ell_q$ is increasing with respect to the ``depth'' $1- q$. See Fig. \ref{fig:andamento_ellq_3r} and Fig.s \ref{fig:figure_01}, \ref{fig:figure_02}, \ref{fig:figure_03} for various instances of the disks of indistinguishable perturbations starting from various values $r$ of the resolution limit in the center. 
\end{remark}

\begin{center}
\begin{figure}[H]
\centering
\includegraphics[scale=0.40]{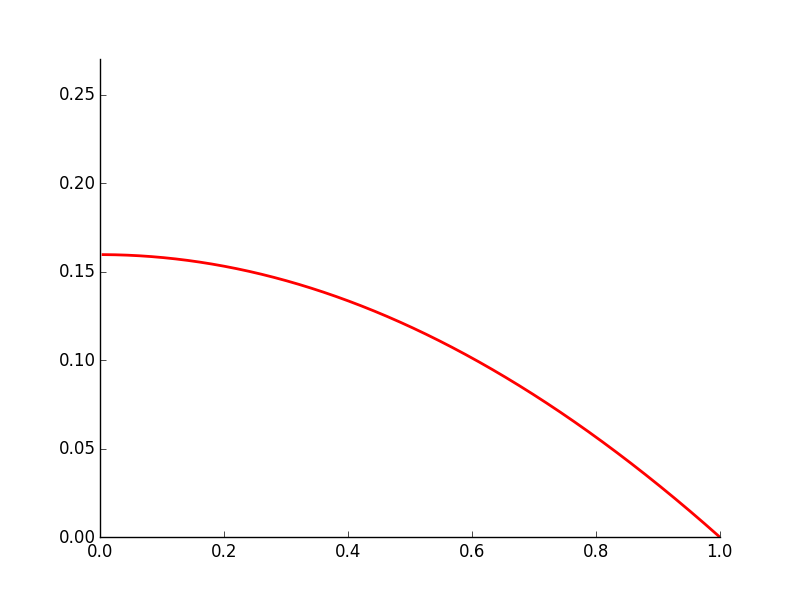}
\caption{$\ell_q$, as function of $q$, with $K = 10^2$ and $\varepsilon_0 = 10^{-1}$\label{fig:andamento_ellq_3r}}
\end{figure}

\begin{figure}[H]
\centering
\includegraphics[scale=0.29]{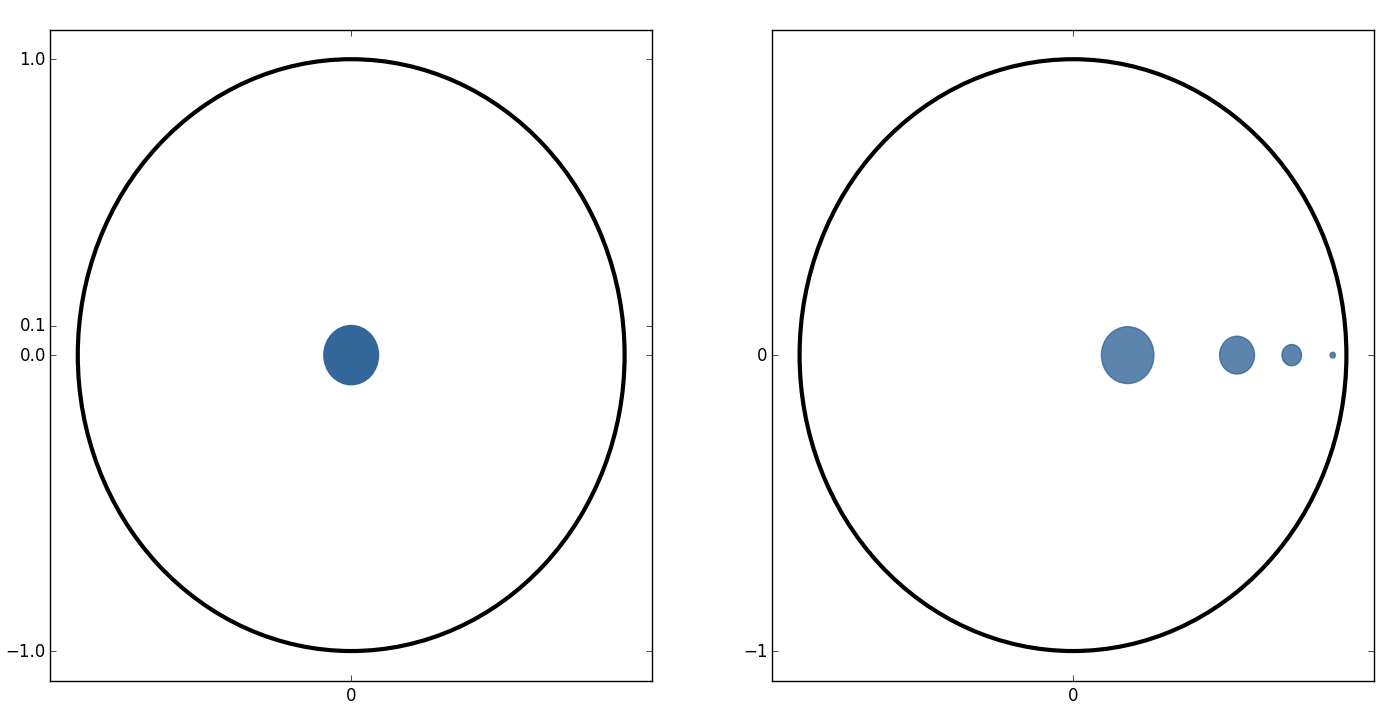}
\caption{$r = 0.1$ \label{fig:figure_01}}
\end{figure}

\begin{figure}[H]
\centering
\includegraphics[scale=0.29]{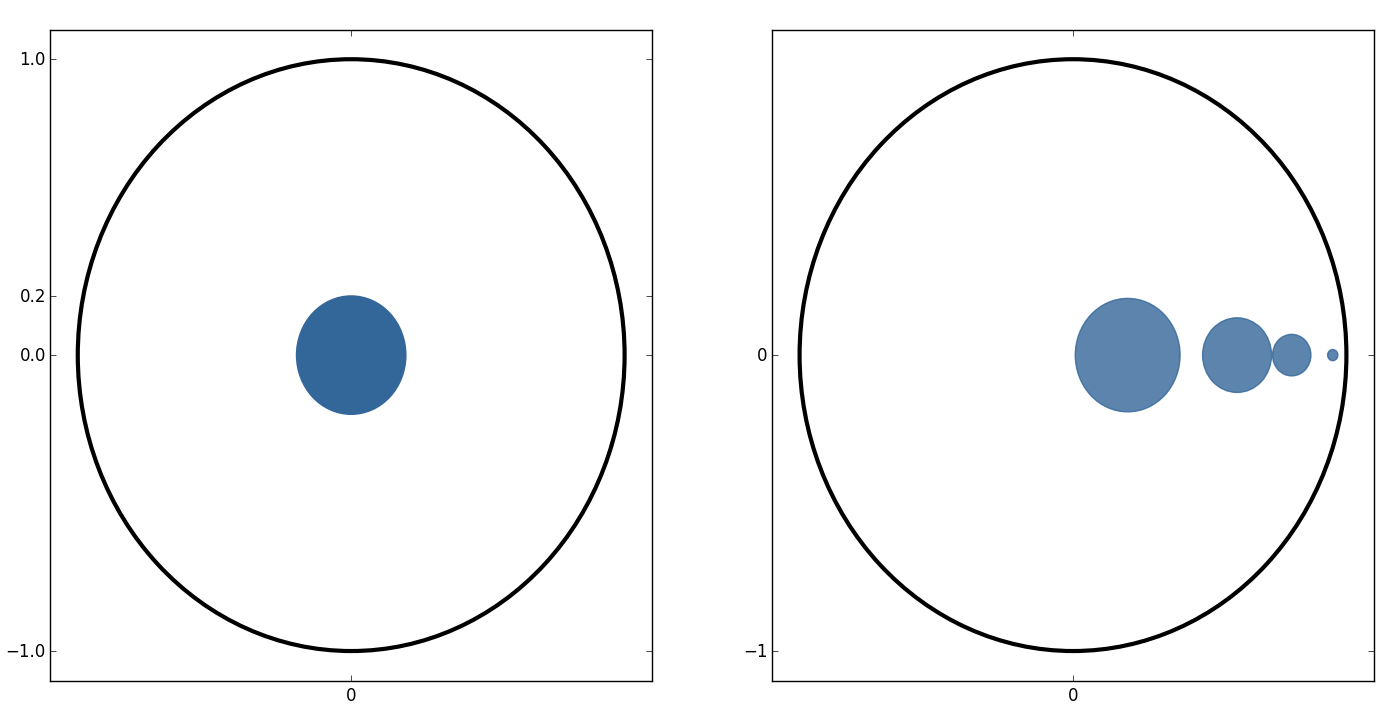}
\caption{$r = 0.2$ \label{fig:figure_02}}
\end{figure}

\begin{figure}[H]
\centering
\includegraphics[scale=0.29]{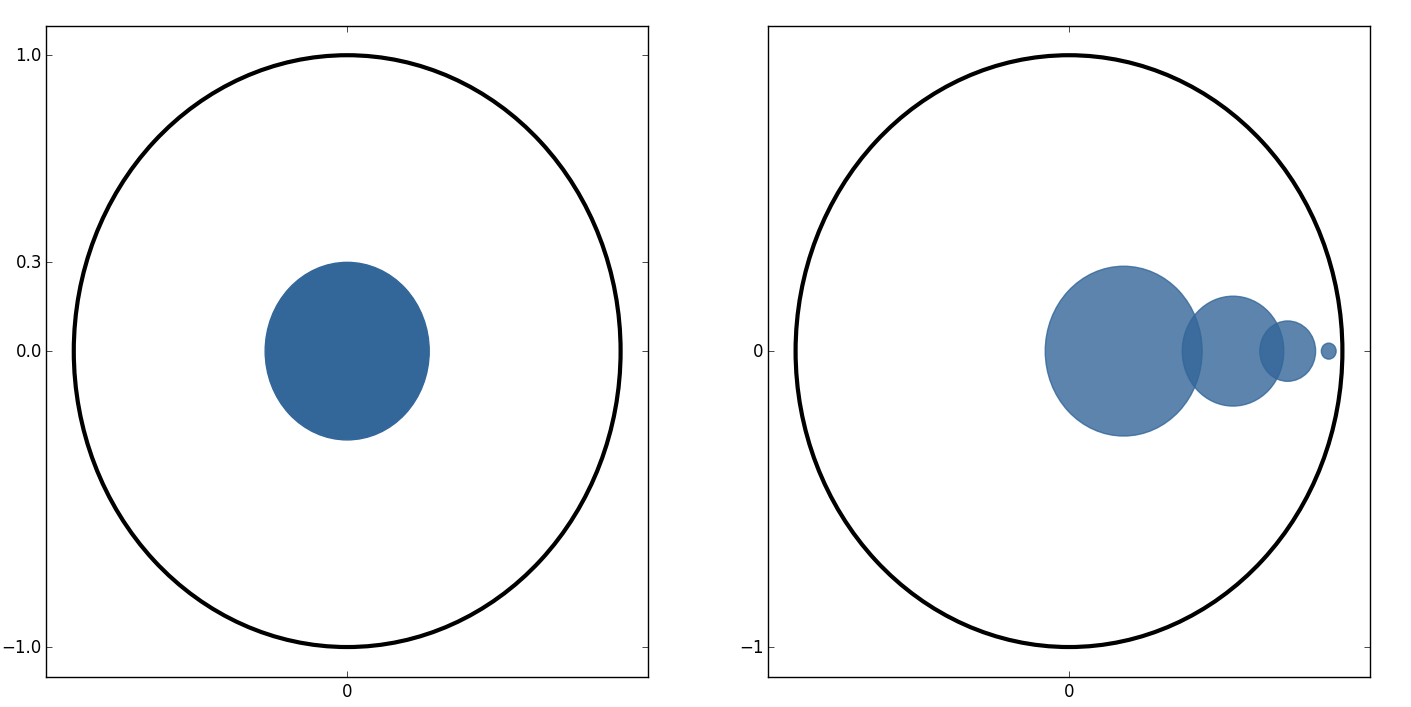}
\caption{$r = 0.3$\label{fig:figure_03}}
\end{figure}
\end{center}

\begin{remark} \label{prop:asymp_bhv} For \eqref{eq:resolution_limit_disk_q} we have the following asymptotic behaviour
 \begin{equation} \label{eq:asymp_bhv} \ell_q(\varepsilon_0 , K) = \dfrac{2 \ell_0}{1 + \ell_0^2} (1 - q) + o (1 - q)  \quad \text{ as } q \to 1 , \end{equation}
 where $\ell_0 = \ell_0 (\varepsilon_0 , K)$ is given in \eqref{eq:eq_risol-bound}.
\end{remark}

Now, our aim is to provide an explicit formula of the resolution limit in the case of the half plane $\mathbb{H}^+$. 

\begin{proposition} \label{prop:half_plane_mobius} Given $r \in (0,1)$, $q \in (0, + \infty)$ and $\alpha \in \RR$, there exists a Möbius transformation $f :  \overline{\mathbb{H}^+} \arr \overline{B_1(0)}$ such that $f(B_{\rho}(\alpha + iq)) = B_r(0)$, where \begin{equation} \label{eq:relazione-rhoqr} \rho = \dfrac{2 q r}{1 + r^2} . \end{equation}
\end{proposition}
\begin{proof} Up to rotations in the target, the generic Möbius transformation which maps $\mathbb{H}^+$ into $B_1(0)$ is given by
\[ f_a(z) = \frac{z - a}{z - \overline{a}} ,\]
for any $a  \in \mathbb{H}^+$. We have 
\[ \left | f_a(z) \right | = r \quad \text{ if and only if } \quad  \left | \frac{z - a}{z - \overline{a}} \right |^2 = r^2.\]
That is: $f_a^{-1}$ maps $B_r(0)$ onto $B_\rho (\alpha +i q) \subset \mathbb{H}^+$ with $q, \rho$ given by
\[ \begin{cases} q =  \beta \dfrac{1 + r^2}{1 - r^2}, \\ \\ \rho = \beta \dfrac{2  r}{1 - r^2},\end{cases} \]
where $a = \alpha + i \beta$. Viceversa, given $q$ and $r$, we can solve for $\beta$ and obtain
\[ \begin{cases}  \beta = q  \dfrac{1 - r^2}{1 + r^2} ,\\ \\ \rho = \dfrac{2 q r}{1 + r^2}.\end{cases} \]
and \eqref{eq:relazione-rhoqr} follows.
\end{proof}

\begin{theorem}[Depth dependent resolution in a half plane] Let $\Omega = \mathbb{H}^+$. The resolution limit at level $\varepsilon_0 > 0$, relative to the depth level $q$ (relative to any point in the half plane whose distance from $\partial \mathbb{H}^+$ is $q > 0$) is given by
\begin{equation}\label{eq:eq_risol-bound-semipiano}  \widetilde{\ell}_q = \dfrac{2 q \ell_0 }{1 + \ell_0^2} ,\end{equation}
where $k$ is the constant introduced in \eqref{eq:k} and $\ell_0$ as in \eqref{eq:eq_risol-bound}.
\end{theorem}
\begin{proof} Immediate consequence of Corollary \ref{cor:cor_conformal-map-norm}, Theorem \ref{eq:eq_risol-bound} and Proposition \ref{prop:half_plane_mobius} .
\end{proof}
See Fig.s \ref{fig:figure_s005rit}, \ref{fig:figure_s01rit}, \ref{fig:figure_s015rit}.  For better interpreting the half plane as a 2D-model of the underground, the $y$-axis is oriented downwards.

\begin{figure}[H]
\centering
\includegraphics[scale=0.29]{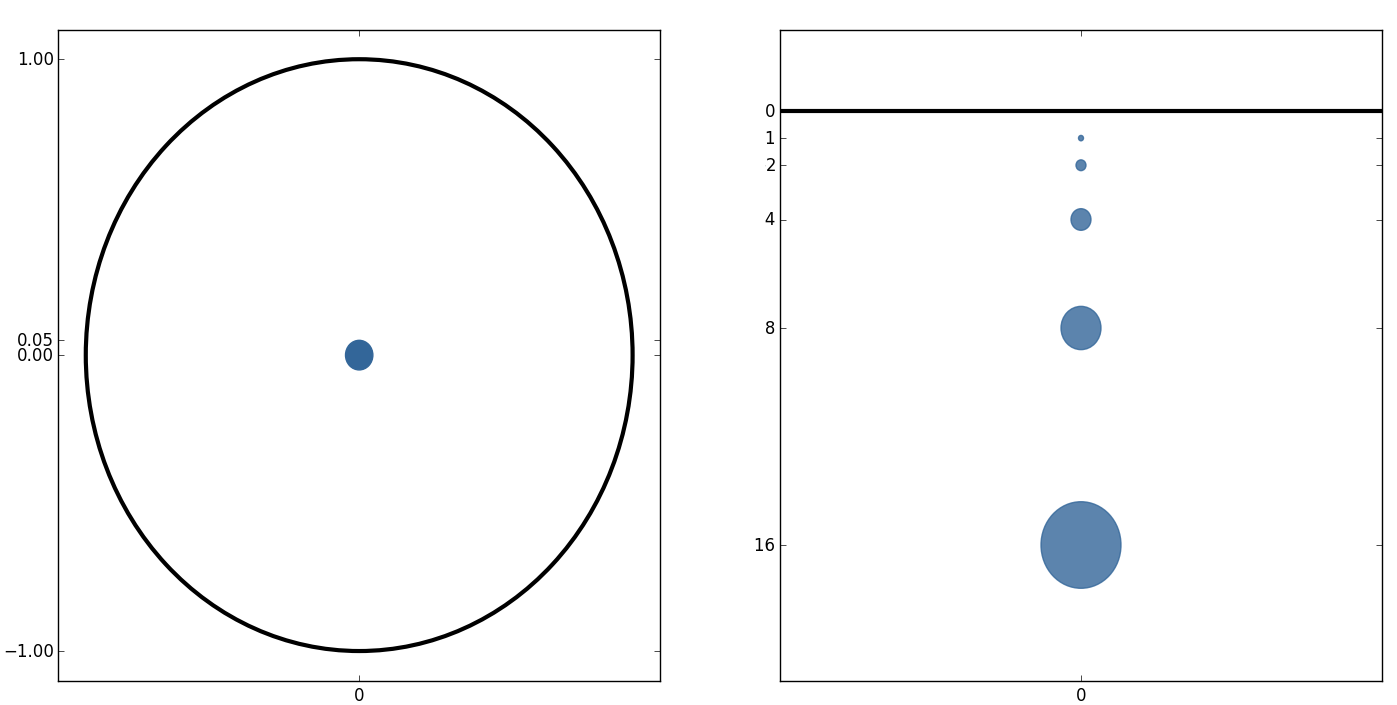}
\caption{$r = 0.05$\label{fig:figure_s005rit}}
\end{figure}

\begin{figure}[H]
\centering
\includegraphics[scale=0.29]{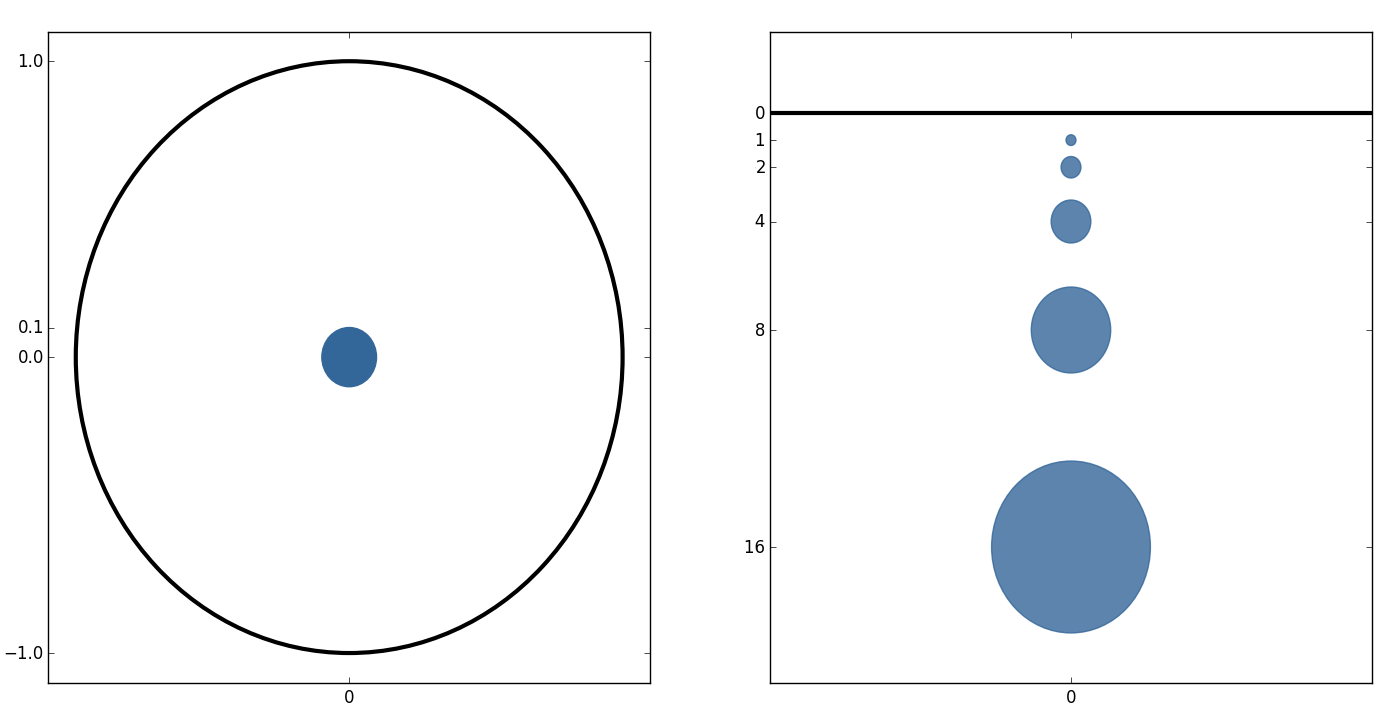}
\caption{$r = 0.1$\label{fig:figure_s01rit}}
\end{figure}

\begin{figure}[H]
\centering
\includegraphics[scale=0.29]{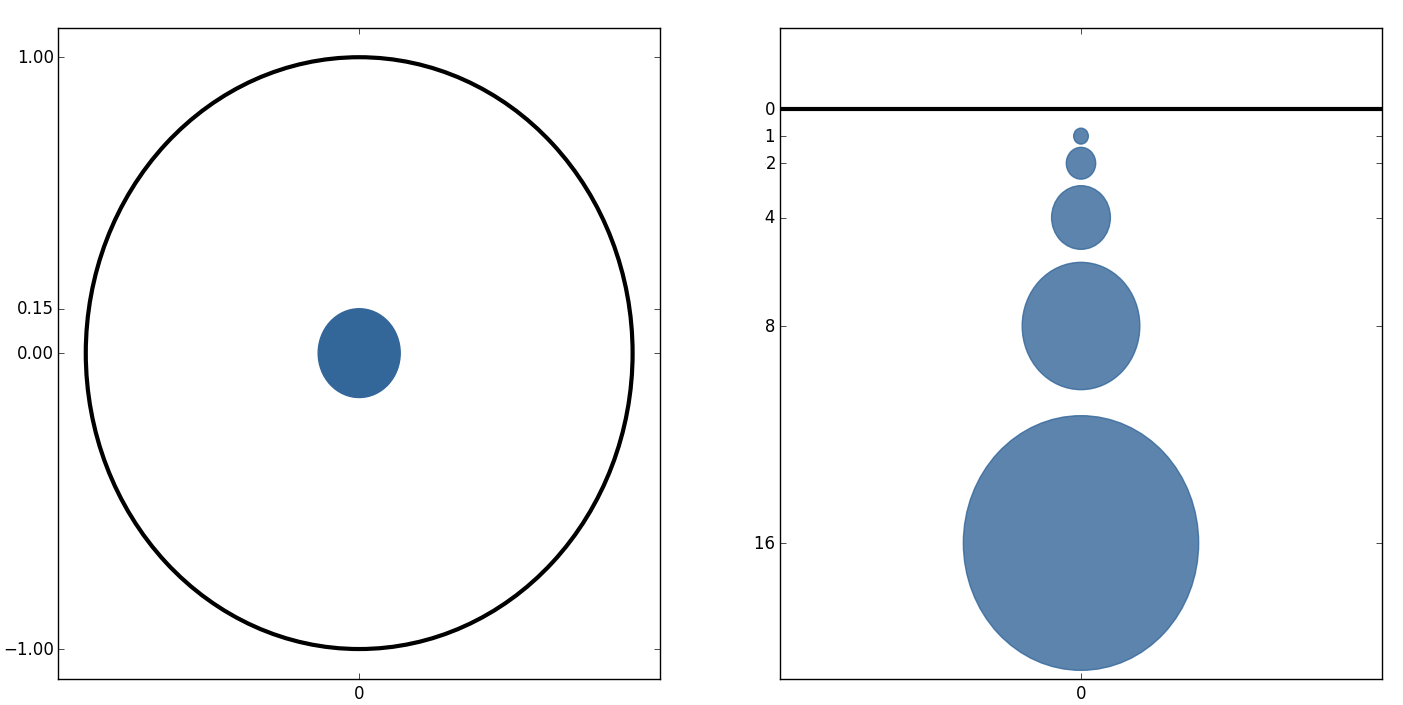}
\caption{$r = 0.15$\label{fig:figure_s015rit}}
\end{figure}

\begin{remark}
In the case of the half plane it is evident from \eqref{eq:eq_risol-bound-semipiano} that the resolution diverges linearly with respect to the depth, that is when $q \to + \infty$. See Fig. \ref{fig:semipiano_cono}.
\end{remark}
\begin{figure}[H]
\centering
\includegraphics[scale=0.29]{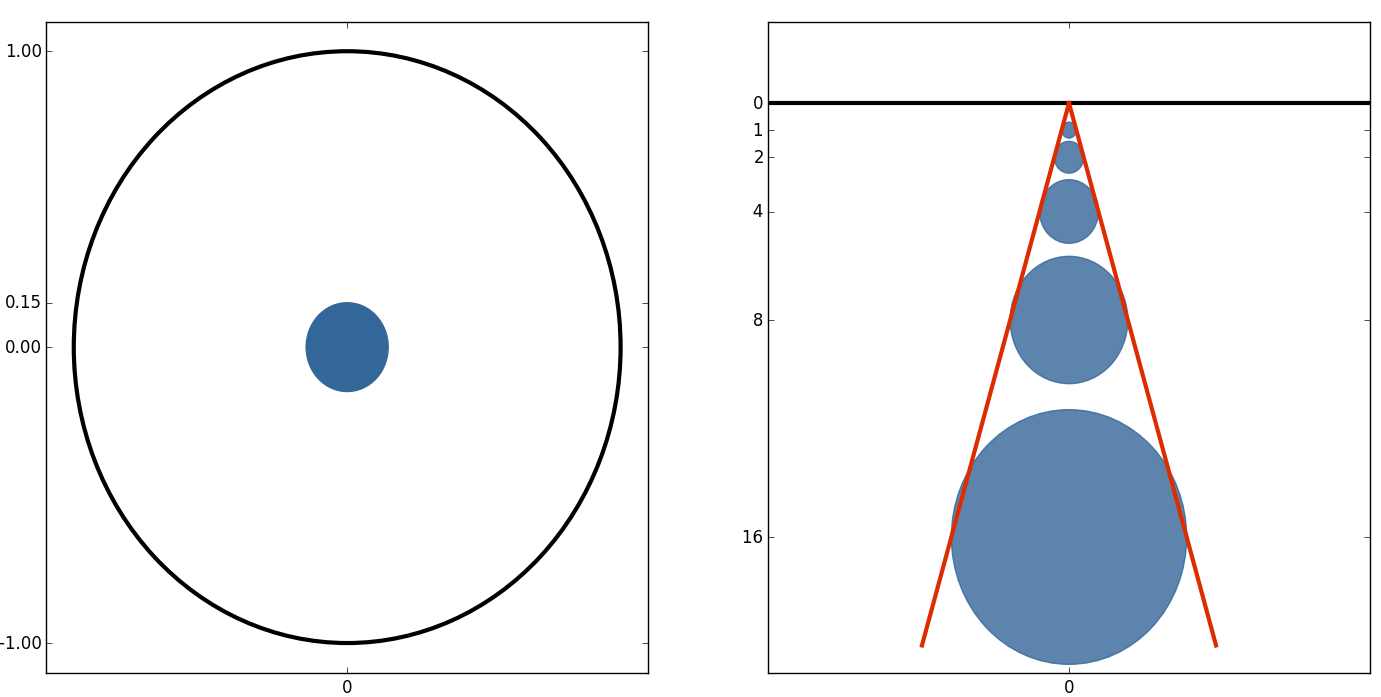}
\caption{The resolution cone, with $r = 0.15$. \label{fig:semipiano_cono}}
\end{figure}

Finally, we observe the following asymptotic behaviour of $\widetilde{\ell}_q $ as function of $\varepsilon_0$.

\begin{remark} Given $q > 0$ we have
\begin{equation} \widetilde{\ell}_q = \frac{q}{\sqrt{k}}  \left ( \sqrt{\epsilon_0} + O(\epsilon_0^{3/2})  \right ) \quad \text{ as } \epsilon_0 \to 0^+. \end{equation} 
\end{remark}


\begin{thebibliography}{999}

\bibitem{A88} G. Alessandrini, Stable determination of conductivity by boundary measurements, App. Anal. \textbf{27} (1988) 153--172.

\bibitem{A89} G. Alessandrini, Remark on a paper by H. Bellout and A. Friedman. Boll. Un. Mat. Ital. A (7), 3(2) (1989) 243--249.

\bibitem{AR98} G. Alessandrini, E. Rosset, The inverse conductivity problem with one measurement: bounds on the size of the unknown object. SIAM J. Appl. Math., 58(4) (1998)1060--1071.

\bibitem{ARS00} G. Alessandrini, E. Rosset, J. K. Seo, Optimal size estimates for the inverse conductivity
problem with one measurement. Proc. Amer. Math. Soc., \textbf{128}(1) (2000) 53--64.

\bibitem{A07} G. Alessandrini, Open issues of stability for the inverse conductivity problem, J. Inv. Ill-Posed
 Problems, \textbf{15} (2007) 1--10.
 
 \bibitem{AIM} K. Astala, T. Iwaniec, G. Martin, Elliptic partial differential equations and quasiconformal mappings in the plane. Princeton Mathematical Series, 48. Princeton University Press, Princeton, NJ, 2009.
 
 \bibitem{Bro} R. Brown, Recovering the conductivity at the boundary from the local
Dirichlet-to-Neumann map: A pointwise result. J. Inverse and Ill-Posed Prob.
9 (2001) 567--574.

\bibitem{Ca} A. P. Calder\'{o}n,
On an inverse boundary value problem, Seminar on Numerical
Analysis and its Applications to Continuum Physics (Rio de
Janeiro, 1980),   65--73, Soc. Brasil. Mat., Rio de Janeiro, 1980.
Reprinted in: Comput. Appl. Math.  \textbf{25}   no. 2--3 (2006)
133--138.

\bibitem{Fa} D. Faraco, M. Prats,
Characterization for stability in planar conductivities. arXiv:1701.06480.

\bibitem{GK} H. Garde, K. Knudsen, Distinguishability revisited: depth dependent bounds on reconstruction quality in electrical impedance tomography. arXiv:1602.03785.

\bibitem{Ike} M. Ikehata, Size estimation of inclusion, J. Inverse Ill-Posed Probl., \textbf{6}(2) (1998)127--140.

\bibitem{Isa} D. Isaacson, Distinguishability of conductivities by electric current computed tomography, IEEE Trans. Med. Imag., vol. MI-5 (1986) 91-95.

\bibitem{IC} D. Isaacson, M. Cheney, Distinguishability in impedance imaging, IEEE Trans Biomed Eng. 1992 Aug;39(8) (1992) 852--860.

\bibitem{KSS} H. Kang, J. K. Seo, D. Sheen, The inverse conductivity problem with one measurement: stability and estimation of size. SIAM J. Math. Anal., \textbf{28}(6) (1997) 1389--1405.

\bibitem{Man} N. Mandache, Exponential instability in an inverse problem for the Schr\"{o}dinger equation, Inverse Problems \textbf{17} (5) (2001) 1435--1444.

\bibitem{NUW} S. Nagayasu, G. Uhlmann, J-N. Wang, A depth-dependent stability estimate in electrical impedance tomography, Inverse Problems \textbf{25}    (7) (2009) 075001, 14 pp.

\bibitem{U} G. Uhlmann,
Electrical impedance tomography and Calder\'{o}n's problem
(topical review), Inverse Problems, \textbf{25} (12) (2009), 123011, 39 pp.

\bibitem{Sy-U-2} J. Sylvester and G. Uhlmann, Inverse boundary value problems at the boundary - continuous dependence, Comm. Pure Appl. Math. \textbf{41} (2) (1988) 197--219.

\end{thebibliography}
\end{document}